\documentclass[11pt,letterpaper,reqno,
]{amsart}

\usepackage[margin=0.8in]{geometry}
\usepackage{amsmath,amssymb,amsthm,latexsym}
\usepackage{graphicx}
\usepackage{float}
\usepackage{url}

\usepackage{tkz-euclide}

\newtheorem*{theorem*}{Theorem}

\newtheorem{lemma}{Lemma}

\usepackage[]{hyperref}

\begin{document}

\title[Morley Trisectors \& Law of Sines]{Morley Trisectors and the Law of Sines \\  with Reflections\\}

\author{ Eric L. Grinberg \& Mehmet Orhon }
\address{UMass Boston \&  University of New Hampshire}
\email{eric.grinberg@umb.edu and mo@unh.edu}


\begin{abstract}
We invoke the law of sines to prove  Morley's trisector theorem. Though the sinusoidal function appears, the proof is safe for trigonometric distancing.
\vspace{-0.3truein}
\end{abstract}

\maketitle

\vspace{-10pt}
\section{Telegraphic Introduction.}

We will provide a proof of the celebrated \emph{Morley trisector theorem}, relying on the law of sines for the crux of the argument. Although we employ the sinusoidal function,  the proof is deemed safe for those with trigonometric restrictions.

\begin{figure}[h!] 
\centering
\includegraphics[scale=0.6]{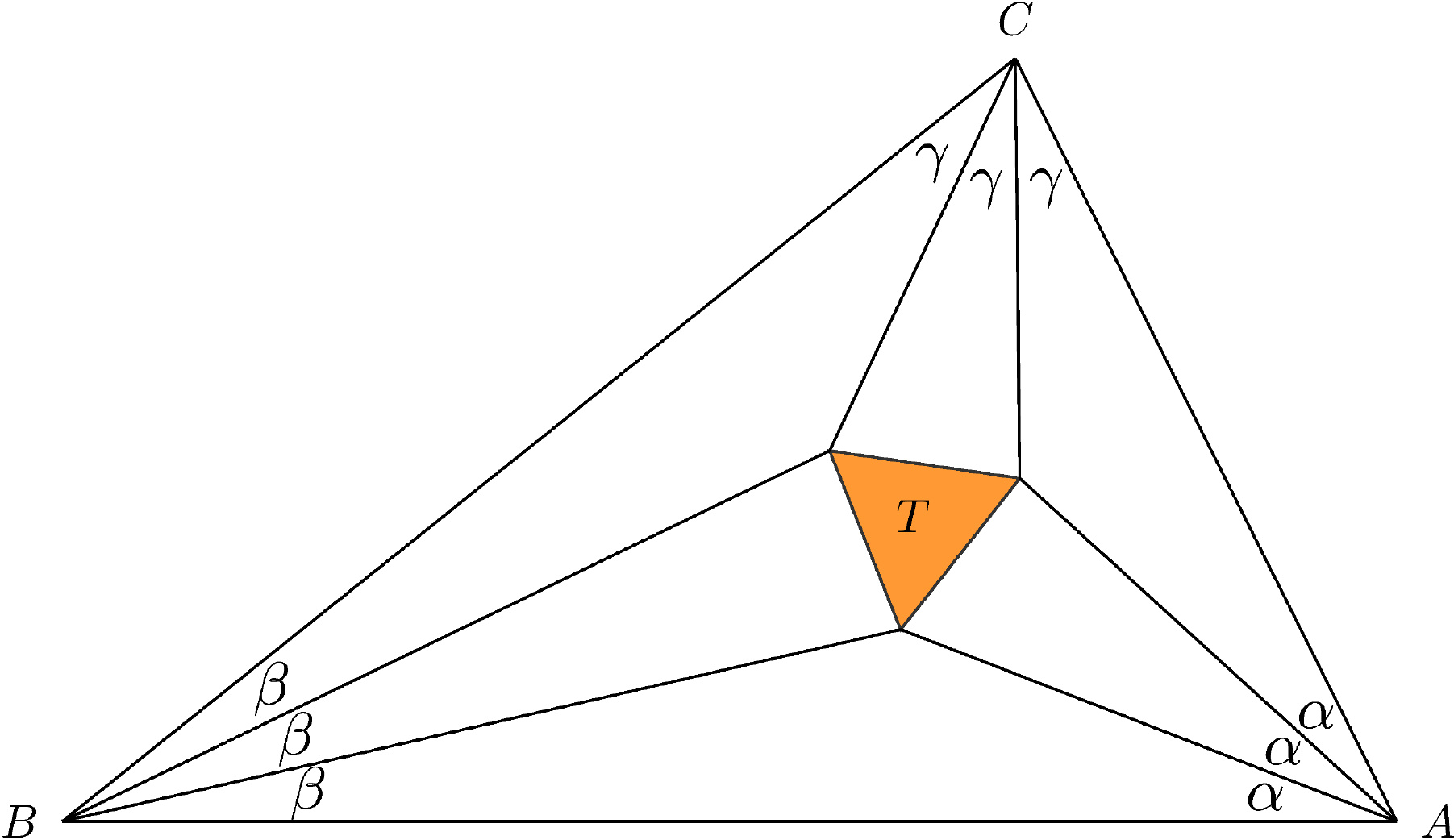}
 \caption{The Morley triangle $T$.}
 \label{fig:The Morley Triangle}
 \end{figure}

We begin with a triangle $\triangle ABC$ having interior angles $3 \alpha, 3 \beta, 3 \gamma$. By trisecting each interior angle and then joining nearby rays of trisection, we obtain three points inside $\triangle ABC$. These interior points form a triangle $T$ (see Figure \ref{fig:The Morley Triangle}) and we wish to show that $T$ is equilateral.  The triangle $\triangle ABC$, which we'll call the native triangle, will figure only fractionally in our discussion, until it makes its full return.

\section{A Nonlinear Lemma.} 

\begin{lemma}
\label{nonlinear_lemma}
Let $\alpha , \beta$ and $\alpha^\prime , \beta^\prime$ be two pairs of proper acute angles. Assume they have the same sum and the same ratio of sines:
\[
\alpha+\beta =\alpha^\prime +  \beta^\prime 
\,  ; \quad
\frac{\sin ( \alpha ) }{ \sin ( \beta ) } = \frac{\sin ( \alpha^\prime ) }{ \sin ( \beta^\prime ) }
\text{.} \]
Then the corresponding angles are equal:
\[ 
\alpha= \alpha^\prime  \,  ; \quad \beta=\beta^\prime
\text{.} \]
\end{lemma}
 Gazing at the hypothesis, we will denote the relation to the left of the semi-colon the \emph{linear relation} and denote the other as the \emph{nonlinear} relation.
\begin{proof}
Suppose that $\alpha < \alpha^\prime$. By the linear relation in the hypothesis, we then have 
\( \beta > \beta^\prime \).  Noting that the sine function is monotone increasing in the domain of acute angles,  we have 
 \begin{equation}
 \label{one}
0 <  \sin ( \alpha ) < \sin ( \alpha^\prime ) \, ; \quad  \sin ( \beta ) > \sin ( \beta^\prime ) > 0 \,
 \text{.} 
 \end{equation}
Taking the reciprocal version of the right portion of \eqref{one} and multiplying, we obtain
\[
\frac{ \sin ( \alpha ) }{ \sin ( \beta )} < \frac{ \sin ( \alpha^\prime ) }{ \sin ( \beta^\prime )} 
\text{,}\]
which contradicts the nonlinear portion of the hypothesis. Next,  a similar argument excludes the possibility $\alpha > \alpha^\prime$. So $\alpha=\alpha^\prime$ and hence $\beta=\beta^\prime$.
\end{proof}

\section{Dashed Hopes.}

We now consider an abstract equilateral triangle $W$ (in \emph{the cloud}). 
Draw $W$ with one side close to horizontal at the top and subtend angles of $\gamma +60^\circ $ to the left and right vertices of this nearly horizontal side. Form left and right triangles with the subtended angles, so that the new far angles are $\beta$ and $\alpha$ and the new angles at the lower vertex of $W$ are $\alpha+60^\circ$ and $\beta+60^\circ$, as in Figure \ref{fig:An abstract triangle}. Since $\alpha + \beta + \gamma=60^\circ$, this is clearly possible.

\begin{figure}[h!] 
\centering
\includegraphics{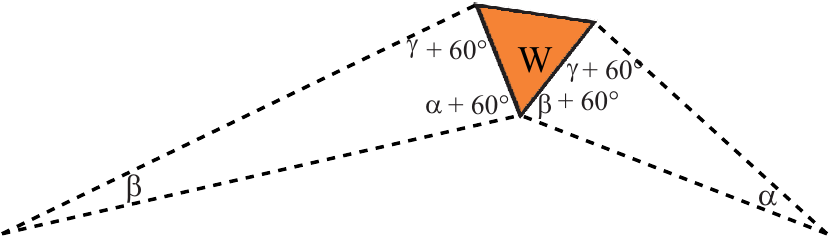}
 \caption{An abstract triangle.}
 \label{fig:An abstract triangle}
 \end{figure}

Now introduce a horizontal line segment connecting the bottom vertices of the two dashed triangles.
 We define angles $\beta^\prime$ and $\alpha^\prime$ associated with $\beta$ and $\alpha$ by this new segment, as illustrated in Figure \ref{fig:Equality of angle sums}, with the label $W$ removed to reduce clutter. Gazing around the top vertex of the newly formed triangle, we see that the top angle, added to $\alpha+60^\circ , \beta+60^\circ$, and $60^\circ$, must total $360^\circ$. Since $\alpha+\beta+\gamma = 60^\circ$, the top angle must be $\gamma + 120^\circ$.  Thus the labels $\beta^\prime , \alpha^\prime$ are definitional, and $\gamma + 120^\circ$ is consequential.

 \begin{figure}[h!] 
\centering
\includegraphics{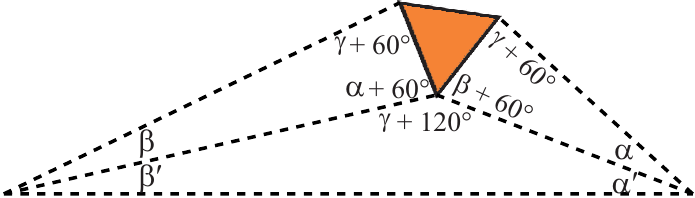}
 \caption{Equality of angle sums.}
 \label{fig:Equality of angle sums}
\end{figure}

 \begin{lemma} In Figure \ref{fig:Equality of angle sums},
 \begin{equation}
 \label{eq:two}
 \alpha + \beta = \alpha^\prime + \beta ^\prime \text{.}  
 \end{equation}
 \end{lemma}
  
 \begin{proof}
  Adding interior angles in the lowest triangle in Figure 3 we have
  \[
  \alpha^\prime + \beta ^\prime + \gamma + 120^\circ=180^\circ \text{.}
  \]
  Since $\alpha, \beta, \gamma$ are trisected editions of the interior angles of our native triangle, we have
  \[
  \alpha + \beta + \gamma = 60^\circ \text{.}
  \]
  Combining the last two equations we have 
  \(  \alpha + \beta = \alpha^\prime + \beta ^\prime \) .
 \end{proof}
 We now sleepily label with $Z$'s the $3$ sides of the abstract equilateral triangle. In the lowest triangle of Figure \ref{fig:Equality of angle sums} we label with $X$ the side opposite $\alpha^\prime$ and we label with $Y$ the side opposite $\beta^\prime$, obtaining Figure \ref{fig:Equality of angle pairs}.

\begin{figure}[h] 
\centering
\includegraphics[scale=0.65]{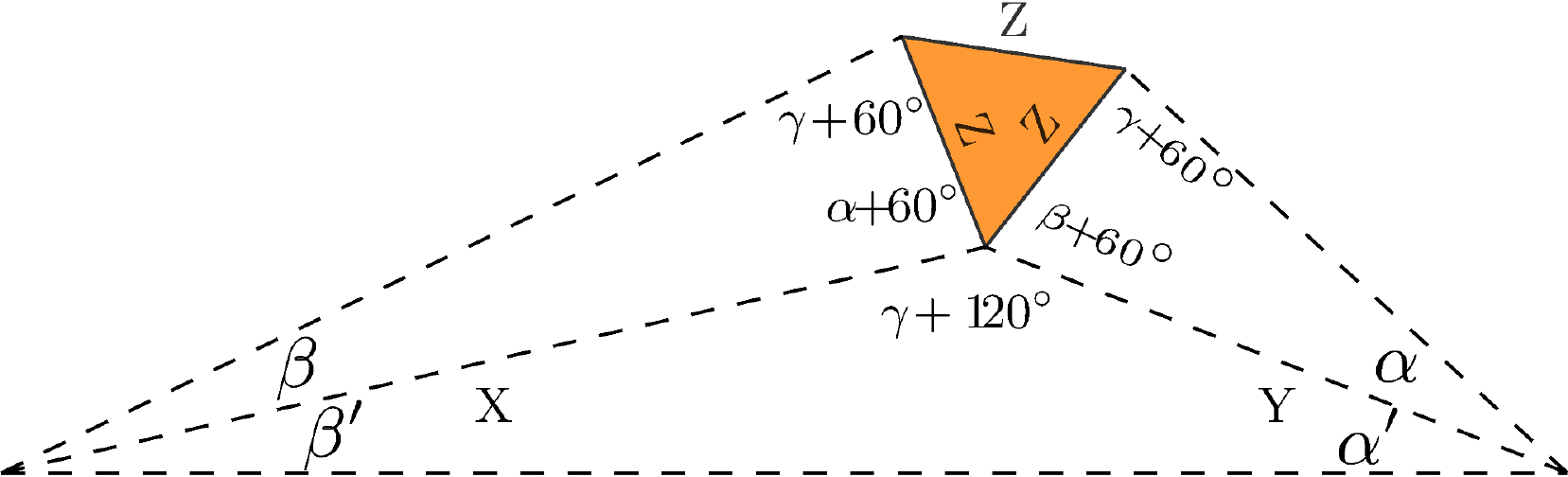}
 \caption{Equality of angle pairs.}
 \label{fig:Equality of angle pairs}
\end{figure}

 \begin{lemma}  The corresponding companion angles introduced by the lowest triangle are equal. That is, 
 \[
 \alpha = \alpha^\prime \, ; \quad \beta = \beta^\prime \text{.}
 \]
 \end{lemma}
\begin{proof}
By the law of sines in the upper-left dashed triangle of Figure \ref{fig:Equality of angle pairs}, 
\[ 
\frac { \sin (\beta ) }{ Z} = \frac{ \sin (\gamma+ 60^\circ )}{ X} \text{.}
\]
By the law of sines in the upper-right dashed triangle of Figure \ref{fig:Equality of angle pairs},
\[ 
\frac { \sin (\alpha  ) }{ Z} = \frac{ \sin (\gamma+ 60^\circ )}{ Y} \text{.}
\]
Dividing the last two equations, we obtain $\sin ( \alpha ) / \sin ( \beta ) = X/Y$.  Making a direct  law of sines argument for $\alpha^\prime, \beta^\prime$  using the lowest triangle, we obtain
\begin{equation}
\label{eq:three}
   \frac{ \sin ( \alpha ) }{ \sin ( \beta ) }
= \frac{X}{Y} 
= \frac{ \sin ( \alpha^\prime ) }{ \sin ( \beta^\prime ) }\text{.}  
 \end{equation}
In \eqref{eq:two}  above we established that 
\(
\alpha + \beta = \alpha^\prime + \beta ^\prime
\).
This, together with  \eqref{eq:three} and Lemma \ref{nonlinear_lemma}
completes the proof.

\end{proof}

 \begin{theorem*} The Morley trisector triangle is equilateral.
 \end{theorem*}
 \begin{proof}
 
We have shown that the lowest triangle in Figure \ref{fig:Equality of angle pairs} is similar to the lowest triangle in the native Morley trisector diagram. Let's call this triangle an \emph{outer triangle}. With the same ideas we can construct two additional ``outer" triangles, emanating from the other two vertices of the abstract equilateral triangle $W$, as in Figure \ref{fig:Return of the native triangle}. 

\begin{figure}[h] 
\centering
\includegraphics[scale=0.65]{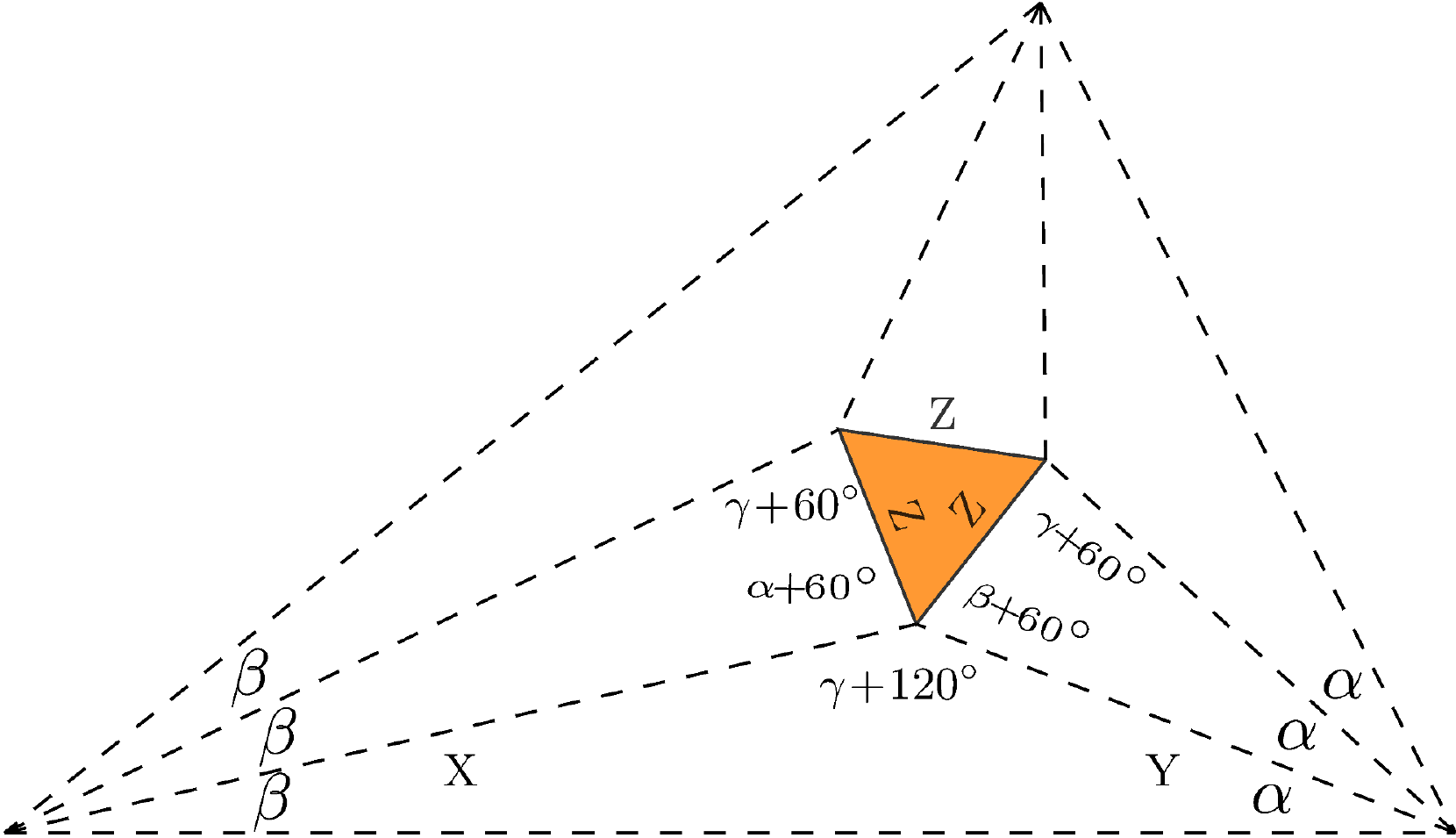}
 \caption{Return of the native triangle.}
 \label{fig:Return of the native triangle}
 \end{figure}

The three outer triangles ``surround" the abstract equilateral  triangle $W$ (in a friendly way). Their longest sides form a large triangle with angles $3\alpha, 3\beta , 3\gamma$. These angles are trisected by our dashed lines and the large triangle is similar to our native triangle $\triangle ABC$. The dilatation (zoom) factor that sends our constructed large triangle to the native triangle $\triangle ABC$ also sends the three constructed outer triangles to corresponding triangles in the original Morley picture, and identifies our abstract equilateral triangle $W$ with the native Morley triangle $T$, which, hence, is also equilateral.
 \end{proof}

 \section{Reflections.}

 We contend that the proof above is trig-light. For the law of sines emanates from basic geometry, as does the monotonicity of the sine function over acute angles. In contrast, see \cite{1965} for a (justifiably) trig-heavy proof of an analog of Morley's theorem.  This is not to disparage trig-heavy arguments, which can be surprisingly beautiful \cite{2016}.  Commentary on mathematical surprise and beauty, in the context of Morley's theorem,  is given  in G.-C.~Rota's book \cite{1997}.

 The law of sines, while less ancient than Euclid, is said to be implicit in Ptolemy's work \cite[p.~44]{2011}  and salient in the work of the 13th century mathematician Mohammed al-Tusi. Interestingly, the MathSciNet review of \cite{2011} points to al-Tusi's result that a pair of angles may be determined from their sum and the ratio of their sines. Our nonlinear Lemma 1 is a nondeterministic uniqueness variant.

 Proofs of the Morley trisector theorem abound in the literature. See, for instance \cite[pp.~23--25]{1969} and \cite{1998,  2014, 1900, 1996, 1978, 2015}. For some time, it was thought to be a difficult theorem to prove. (See, e.g., \cite{1998, 2014}.) Then a plethora of ``easy" proofs emerged, and the trend continues. 
 
 Yet, this theorem still does not appear in \emph{Proofs From the Book} \cite{2018}. Then again, neither does the far more senior  Pythagorean theorem, which has several hundred proofs \cite{1968}. Now, the 6th edition of \cite{2018} has six ``Book" proofs of the infinitude of the primes.  Is there a Book proof of the Morley trisector theorem? Perhaps the case of the Pythagorean theorem should be settled first.
 
 \section{Acknowledgements}

We thank everyone who gave us helpful comments and suggestions.
We also thank the authors of \LaTeX  \,  codes for the Morley trisector diagram, \cite{2010,2011b}. Our diagrams are based on their posted source material.

\end{document}